\documentclass[11pt]{amsart}

  \usepackage[utf8]{inputenc}

  \usepackage{amsmath}
  \usepackage{amsfonts}
  \usepackage{amssymb}
  \usepackage{amsthm}

  \usepackage{mathrsfs}
  \usepackage{tikz}

  \usepackage[a4paper, bindingoffset=1cm]{geometry}

  \usetikzlibrary{matrix,arrows,calc}
  
\usepackage{xfrac}
\usepackage{amsmath}
\usepackage{amstext}
\usepackage{amssymb}
\usepackage{amsthm}
\usepackage{mathtools}
\usepackage{mathrsfs}
\usepackage{microtype}
\usepackage{bbm}
\newtheorem{theo}{Theorem}[section]
\newtheorem{lemma}[theo]{Lemma}

\newtheorem{rema}[theo]{Remark}
\newtheorem{prop}[theo]{Proposition}
\newtheorem{cor}[theo]{Corollary} 
\theoremstyle{definition}

\theoremstyle{remark}
 
\newcommand{\matN}{\ensuremath {\mathbb{N}}}
\newcommand{\matQ}{\ensuremath {\mathbb{Q}}}

\newcommand{\matZ} {\ensuremath {\mathbb{Z}}}

\renewcommand{\epsilon}{\ensuremath \varepsilon}
\renewcommand{\bar}[1]{\ensuremath \overline{#1}}

\begin{document}
 \title[]{
\textbf{Sparsity of p-divisible unramified liftings for subvarieties of abelian varieties with trivial stabilizer}
}
\author{DANNY SCARPONI} 
\thanks{Danny Scarponi, Fakultät für Mathematik, Universität Regensburg, 93040 Regensburg.
Danny.Scarponi@mathematik.uni-regensburg.de}





\maketitle

 
  
  
  
\begin{abstract}
By means of the theory of strongly semistable sheaves and of the theory of the Greenberg transform, we generalize to higher dimensions a 
result on the sparsity of $p$-divisible unramified 
liftings which played a crucial role in Raynaud's proof of the Manin-Mumford conjecture for curves. We  also give a bound 
for the number of irreducible components of the first critical scheme of   subvarieties of an abelian variety which are complete intersections.
 \tableofcontents


\end{abstract}

\section{Introduction}

 The Manin-Mumford conjecture is a significant question concerning the intersection of a subvariety $X$ of an abelian variety $A$ with the group of torsion points of $A$. Raised independently by Manin and Mumford, the conjecture was originally formulated in the case of curves. Suppose that $A$ is an abelian 
 variety over a number field $K$ and that $C$ is a 
 smooth  subcurve of $A$ of genus at least two. 
Then  only finitely many torsion points of $A(\bar K)$ lie in $C$. 
In 1983, Raynaud proved this conjecture and generalized it to higher dimensions: if $A/K$ is as above and $X/K$ is a smooth subvariety of $A$ which does not contain any
translate of a non-trivial abelian subvariety,
then 
the set of torsion points of $A(\bar K)$ lying in $X$ is finite (cf. \cite{courbes}, \cite{rayn2}). 

%
Let us fix $K, X$ and $A$ as above. 
%
Let 
$U$ be a nonempty open subscheme of $\text{Spec} \ \mathcal{O}_K$ not containing any ramified prime and such that $A/K$ extends to an abelian scheme $ \mathcal{A}/U$ and $X$ extends
       to a smooth closed integral subscheme $\mathcal{X}$ of $\mathcal{A}$. For any $\mathfrak{p}\in U$, let $R$ (resp. $R_n$) be the ring of Witt vectors (resp. of length $n+1$) 
       with coordinates in the algebraic closure 
$\bar{k(\mathfrak{p})}$ of the residue field of $\mathfrak{p}$. Recall that $R$ is a DVR with maximal ideal generated by $p$ such that $R_0=R/p=\bar{k(\mathfrak{p})}$. 
Denote by $X_{\mathfrak{p}^n}$ (resp. $A_{\mathfrak{p}^n}$) the $R_n$-scheme 
$\mathcal{X} \times_U \text{Spec} \ R_n$ (resp. $\mathcal{A} \times_U \text{Spec} \ R_n$) and consider 
 the reduction map 
\begin{equation}\label{sparcur} p A_{\mathfrak{p}^1}(R_1)\cap
 X_{\mathfrak{p}^1}(R_1)\rightarrow X_{\mathfrak{p}^0}(R_0). \end{equation} 
In \cite{courbes}, it is showed that, if $X$ is a curve, the image of (\ref{sparcur}) is not Zariski dense in
$X_{\mathfrak{p}^0} $, i.e. it  is a finite set. This local result is crucial in Raynaud's proof of the Manin-Mumford
conjecture for curves, since it easily implies that only finitely many prime-to-$p$ torsion points of $A(\bar K)$ lie on $X$ (cf. Théorème II in 
\cite{courbes}). 
 
It is quite natural to expect that a similar result  also holds in higher dimensions. More explicitely, one can ask: is it  true 
that, if a smooth subvariety $X$ of $A$ does  not contain any translate of a non-trivial abelian subvariety,
the image of (\ref{sparcur})
is not Zariski dense? In this paper we give a positive answer to this question (cf. Theorem \ref{torre}). 
\begin{theo}[Sparsity of $p$-divisible unramified liftings] \label{torreintro}
Suppose that $X$ has trivial stabilizer. For all $\mathfrak{p}\in   U$  above a prime  $p>(\dim X)^2 \deg(\Omega_X) $ and such that  $X_{\mathfrak{p}^0} $ has trivial 
stabilizer, the image of 
\[
  p A_{\mathfrak{p}^1}(R_1)\cap
 X_{\mathfrak{p}^1}(R_1)\rightarrow X_{\mathfrak{p}^0}(R_0)
\]
is not Zariski dense in $X_{\mathfrak{p}^0}$.
\end{theo}
Here $\deg(\Omega_X)$ refers to the degree of the cotangent bundle $\Omega_X$ computed with respect to any fixed very ample line bundle on $X$. 

Notice that if  $X$ does not contain any  translate of a non-trivial  abelian subvariety, then it has finite stabilizer. Therefore,
 replacing $A$ and $X$ with their quotients by the stabilizer of $X$, one can suppose to be in the case of a trivial   stabilizer (cf. beginning of the next section for the definition 
of stabilizer). 

A different generalization of Raynaud's local result was  given in \cite{MMML}, where R\"ossler proved that, if the torsion points 
of $\mathcal{A}(\overline{\text{Frac}(R)})$ are not dense in $\mathcal{X}(\overline{\text{Frac}(R)})$, then 
for $m$ big enough the image of 
\begin{equation}\label{high}
 p^m A_{\mathfrak{p}^m}(R_m)\cap
 X_{\mathfrak{p}^m}(R_m)\rightarrow X_{\mathfrak{p}^0}(R_0)
\end{equation}
is not Zariski dense in $X_{\mathfrak{p}^0}$ (cf. Th. 4.1 in \cite{MMML}). Theorem \ref{torreintro} makes R\"ossler's result effective, showing that if the stabilizer of $X$ is trivial, then 
it is sufficient
to consider the map (\ref{high}) for 
$m=1$.

The proof of Theorem \ref{torreintro} strongly relies on   R\"ossler's paper  \cite{rosstrongly} and is done by contradiction.
First we use some basic properties of the  Greenberg transform to show  that, 
if the image of (\ref{sparcur}) is Zariski dense in $X_{\mathfrak{p}^0}$, then the absolute Frobenius  $ F_{X_{\mathfrak{p}^0}}:X_{\mathfrak{p}^0}\rightarrow X_{\mathfrak{p}^0} $ 
lifts to an endomorphism of $ X_{\mathfrak{p}^1} $. A well-known consequence of
this liftability 
is  the existence of a map of sheaves of differentials 
$F_{X_{\mathfrak{p}^0}}^*\Omega_{X_{\mathfrak{p}^0}}\rightarrow \Omega_{X_{\mathfrak{p}^0}}$
 which is nonzero.
In the case that  $X$ is a curve, such a map cannot exist, 
since $\deg(F_{X_{\mathfrak{p}^0}}^*\Omega_{X_{\mathfrak{p}^0}})$ is strictly bigger than 
$\deg(\Omega_{X_{\mathfrak{p}^0}})$. This simple observation was in fact used    by Raynaud to prove  Lemma I.5.4 in \cite{around}.
By means of the theory of strongly semistable sheaves developed by  R\"ossler in  \cite{rosstrongly}, we show that 
also when $X$ has dimension higher than one, there are no nontrivial maps from $F_{X_{\mathfrak{p}^0}}^*\Omega_{X_{\mathfrak{p}^0}}$
to $\Omega_{X_{\mathfrak{p}^0}}$.  This gives us the wanted contradiction.

In the last section of this paper, we consider subvarieties of abelian varieties which are complete intersections. If $\text{Gr}_1$ denotes the Greenberg transform of level 1 (cf.
section 3), then we know that the first critical scheme \[\text{Crit}^1(\mathcal{X},\mathcal{A}):=[p]_* \text{Gr}_{1}(A_{\mathfrak{p}^{1}})\cap 
       \text{Gr}_{1}(X_{\mathfrak{p}^{1}}) \] is a scheme over $R_0$ such that \[\text{Crit}^1(\mathcal{X},\mathcal{A})(R_0)= 
       p A_{\mathfrak{p}^1}(R_1)\cap
 X_{\mathfrak{p}^1}(R_1).\] 
Using exactly the same technique that allowed Buium 
to give an effective form of the  Manin-Mumford
conjecture in the case of curves (\cite{buium}), we get a bound for the number of irreducible  components of $\text{Crit}^1(\mathcal{X},\mathcal{A}) $ when $X$ is a complete intersection
(not necessarily with trivial stabilizer).
\begin{theo} \label{effintro}
 Let $K$ be a number field, $A/K$ be an abelian variety of dimension $n$ and let $L$ be a very ample line bundle on 
$A$. Let $c\in \matN$ be positive and  let $H_1$, $H_2, ...,H_c \in 
|L|$ be general. Suppose that $X:=H_1 \cap H_2 \cap ... \cap H_c$ is smooth. 
There exists a nonempty open subscheme $V\subseteq \textnormal{Spec} \ \mathcal{O}_K$  (see beginning of Section 6 for its definition) such that if $\mathfrak{p}\in V$,
the number of irreducible components of $\textnormal{Crit}^1(\mathcal{X},\mathcal{A}) $ is bounded by 
\[
 p^{2n}\left(\sum_{h=0}^{n-c}{2n-2c\choose h}{c \choose n-c-h}p^{n-c-h} \right) (L^n)^2.
\]
\end{theo}
Here $(L^n) $ denotes the intersection number of $L$.

We conclude the Introduction with the following remark.  Since the field of definition of points in the prime-to-p torsion $\textnormal{Tor}^p(A(\bar K))$ is unramified 
at $\mathfrak{p}$ and 
 the specialization map $\mathcal{A}(R)\rightarrow 
 A_{\mathfrak{p}^1}(R_1)$ is injective on the prime-to-$p$ torsion, we have an injection
\[
 \textnormal{Tor}^p(A(\bar K))\cap X(\bar K)\subseteq p A_{\mathfrak{p}^1}(R_1)\cap X_{\mathfrak{p}^1}(R_1).
\]
This implies that,
if $X$ is a complete intersection such that $\textnormal{Crit}^1(\mathcal{X},\mathcal{A})(R_0) $ is finite, then the bound in 
Theorem \ref{effintro}  is a bound for the cardinality of $\textnormal{Tor}^p(A(\bar K))\cap X(\bar K)$, i.e  an effective
form of the Manin-Mumford conjecture for the prime-to-$p$ torsion. 
\section{Notations}
 We fix the following notations:
\begin{itemize}
 \item $K$ a number field,
 \item $\bar K$ an algebraic closure of $K$,
 \item $A/K$ an abelian variety,
 \item $X\subseteq A$ a closed integral subscheme, smooth over $K$,
 \item $\text{Stab}_A(X)$ the translation stabilizer of $X$ in $A$, i.e. the closed subgroup scheme of $A$ characterized
 uniquely by the fact that for any $K$-scheme $S$ and any morphism $b: S\rightarrow A$, translation by $b$ on the product
 $A\times_K S$ maps the subscheme $X\times_K S$ to itself if and only if $b$ factors through $\text{Stab}_A(X)$ (for its existence we refer the reader to
 Ex. 6.5(e), Exp. VIII in \cite{schemas}),
 \item $U$ an open subscheme of $\text{Spec} \ \mathcal{O}_K$ not containing any ramified prime and such that $A/K$ extends to an abelian scheme $ \mathcal{A}/U$ and $X$ extends
       to a smooth closed integral subscheme $\mathcal{X}$ of $\mathcal{A}$.
\end{itemize}
For any prime number $p$, any unramified prime $\mathfrak{p}$ of $K$ above $p$ and any $n\geq 0$, we denote by:
\begin{itemize}
 \item $k(\mathfrak{p})$ the residue field $\mathcal{O}_K / \mathfrak{p} $ for $\mathfrak{p}$,
 \item $K_\mathfrak{p}$ the completion of $K$ with respect to $\mathfrak{p}$,
 \item $\widehat{K_\mathfrak{p}^{\text{unr}}}$ the completion of the maximal unramified extension of $K_\mathfrak{p}$,
 \item $R:=W(\bar{k(\mathfrak{p})})$ (resp. $R_n:=W_n(\bar{k(\mathfrak{p})})$) the ring of Witt vectors (resp. the ring of Witt vectors of length $n+1$) 
       with coordinates in $\bar{k(\mathfrak{p})}$. We recall that $R$ can be identified with the ring of integers of $\widehat{K_\mathfrak{p}^{\text{unr}}}$ and 
       $R_0$ with $\bar{k(\mathfrak{p})}$,
 \item $X_{\mathfrak{p}^n}$ the $R_n$-scheme $\mathcal{X} \times_U \text{Spec} \ R_n$ \\
       $A_{\mathfrak{p}^n}$ the $R_n$-scheme $\mathcal{A} \times_U \text{Spec} \ R_n$.
\end{itemize}

\section{The Greenberg transform and the critical schemes}
Now we recall some basic facts about the Greenberg transform (for more details, see \cite{greenb1}, \cite{greenb2} and [\cite{bosch} p. 276-277]).

Fix a prime number $p$ and an unramified prime $\mathfrak{p}$ of $K$ above $p$.

For any $n\geq 0$, the Greenberg transform of level $n$ is a covariant functor $\text{Gr}_n$ from the category of $R_n$-schemes locally of finite type, 
to the category
of $R_0$-schemes locally of finite type. If $Y_n$ is an $R_n$-scheme locally of finite type,
$\text{Gr}_n(Y_n)$ is a $R_0$-scheme with the property \[ Y_n(R_n)=\text{Gr}_n(Y_n)(R_0).\]
More precisely, we can interpret $R_n$ as the set of $\bar{ k(\mathfrak{p})}$-valued points of a ring scheme $\mathscr{R}_n$ over $\bar {k(\mathfrak{p})}$.
For any $R_0 $-scheme $T$, we define $\mathbb{W}_n(T)$ as the ringed space over 
$R_n$ consisting of $T$ as a topological space and of $\textnormal{Hom}_{R_0}(T,\mathscr{R}_n) $ as a structure sheaf. 
By definition $\text{Gr}_n(Y_n)$ represents the functor from the category of schemes over $R_0$ to
the category of sets given by
\[
 T\mapsto \textnormal{Hom}_{R_n}(\mathbb{W}_n(T),Y_n)
\]
where $\textnormal{Hom}$ stands for homomorphisms of ringed spaces. In other words, the functor $\text{Gr}_n $ is right adjoint to the functor $\mathbb{W}_n$.

The functor
$\textnormal{Gr}_n$ respects closed immersions, open immersions, fibre products, smooth, \'etale morphisms and is the identity for $n=0$. Furthermore it sends group schemes over $R_n$ to
group schemes over $R_0$. The canonical morphism $R_{n+1}\rightarrow R_n$
gives rise to a functorial transition morphism $\pi_{n+1}:\text{Gr}_{n+1}\rightarrow \text{Gr}_{n}$.

Let  $Y_{n}$ be a scheme  over $R_{n}$ locally of finite type. Then for any $m< n$ we define 
\[
 Y_m:=Y_{n}\times_{R_{n}}  R_m.
\]
Let us call $F_{Y_0}:Y_0\rightarrow Y_0$ the absolute Frobenius endomorphism of $Y_0$ and $\Omega_{Y_0 / R_0}$ the sheaf of relative differentials.

For any finite rank locally free sheaf $\mathscr{F}$ over $Y_0$   we will write
\[
 V\left(\mathscr{F}\right):=\underline{\rm{Spec}} \left(\textnormal{Sym}\left(\mathscr{F}^\vee\right)\right)
\]
for the vector bundle over $Y_0$ associated to $\mathscr{F}$.

Suppose now that $Y_{n}$ is smooth over $R_{n}$, so that $\Omega_{Y_0/R_0}$ is locally free.
A key result about the Greenberg transform is the following fact (cf. Section 2 in  \cite{greenb2}): 
\[
 \pi_1: \textnormal{Gr}_{1}(Y_{1})\rightarrow  \textnormal{Gr}_{0}(Y_0)=Y_0
\]
is a torsor under the Frobenius tangent bundle
\[
 V\left(F_{Y_0}^{*}\Omega^\vee_{Y_0/R_0}\right).
\]
Let $X$, $A$, $\mathcal{X}$, $\mathcal{A}$ and $U$ be as fixed in the previous section and suppose that $\mathfrak{p}\in U$. We refer the reader to Section II.1 in \cite{around}
for more details on what we will recall from now till the end of the section. For any $n\geq 0$, the kernel of
\[
 \textnormal{Gr}_{n}(A_{\mathfrak{p}^{n}})\rightarrow  \textnormal{Gr}_{0}(A_{\mathfrak{p}^{0}})=A_{\mathfrak{p}^{0}}
\]
is unipotent, killed by $p^n$. Thus, the scheme-theoretic image $ [p^n]_* \textnormal{Gr}_{n}(A_{\mathfrak{p}^{n}})$ of the multiplication by 
$p^n$ in $\textnormal{Gr}_{n}(A_{\mathfrak{p}^{n}}) $ is the
greatest abelian subvariety of  $\textnormal{Gr}_{n}(A_{\mathfrak{p}^{n}}) $ and, since $R_0 $ is algebraically closed, 
$ [p^n]_* \textnormal{Gr}_{n}(A_{\mathfrak{p}^{n}})(R_0)=p^n \textnormal{Gr}_{n}(A_{\mathfrak{p}^{n}})(R_0 )$.

We define the $n$-critical scheme as
\[\text{Crit}^n(\mathcal{X},\mathcal{A}):=[p^n]_* \text{Gr}_{n}(A_{\mathfrak{p}^{n}})\cap 
       \text{Gr}_{n}(X_{\mathfrak{p}^{n}}).
\]
      Notice  that $\text{Crit}^n(\mathcal{X},\mathcal{A})$
       is a scheme over $R_0$ and that $\text{Crit}^0(\mathcal{X},\mathcal{A})=X_{\mathfrak{p}^0}$.

 The transition morphisms $\pi_{n+1}:\textnormal{Gr}_{n+1}(A_{\mathfrak{p}^{n+1}})\rightarrow  \textnormal{Gr}_{n}(A_{\mathfrak{p}^{n}})$ lead to a
 projective system of
  $R_0 $-schemes:
 \[
  \cdots \rightarrow \text{Crit}^2(\mathcal{X},\mathcal{A}) \rightarrow 
  \text{Crit}^1(\mathcal{X},\mathcal{A}) \rightarrow \text{Crit}^0(\mathcal{X},\mathcal{A})=X_{\mathfrak{p}^0}
 \]
 whose connecting morphisms are both affine and proper, hence finite.
In fact, transition morphisms are  affine and  
 the subscheme
$[p^n]_* \textnormal{Gr}_{n}(A_{\mathfrak{p}^{n}})$ is proper, being the greatest abelian subvariety of 
$\textnormal{Gr}_{n}(A_{\mathfrak{p}^{n}})$. 

We shall write $\text{Exc}^n(\mathcal{X}, \mathcal{A})$ for the scheme theoretic image of the morphism
$\text{Crit}^n(\mathcal{X},\mathcal{A}) \rightarrow  X_{\mathfrak{p}^0}$.
\section{The geometry of vector bundles in positive characteristic}
In this section we recall  some results on the geometry of vector bundles in positive characteristic by Langer (\cite{langersem}) and R\"ossler (\cite{rosstrongly}).  These results 
will play a crucial role in the proof of Lemma \ref{p} and Theorem \ref{torre}.

Let us start with some basic definitions and facts regarding semistable sheaves in positive characteristic.

Let $Y$ be a smooth projective variety over an algebraically closed field $l_0$ of positive characteristic. We write as 
before
$\Omega_{Y/l_0}$ for the sheaf of differentials of $Y$ over $l_0$ and $F_Y:Y\rightarrow Y$ for the absolute Frobenius endomorphism of $Y$. 
Now let $L$ be a very ample line bundle on $Y$. If $V$ is a torsion free coherent sheaf on $Y$, we shall write 
\[
 \mu(V)=\mu_L(V)=\textnormal{deg}_L(V)/\textnormal{rk}(V)
\]
for the slope of $V$ (with respect to $L$). Here $\textnormal{rk}(V)$ is the rank of $V$, i.e. the dimension of the stalk of $V$ at the generic point of $Y$. Furthermore,
\[
 \textnormal{deg}_L(V):=\int_Y c_1(V)\cdot c_1(L)^{\textnormal{dim}(Y)-1}
\]
where $c_1(\cdot)$ refers to the first Chern class with values in an arbitrary Weil cohomology theory and the integral $\int_Y $ stands for the push-forward morphism to 
$\textnormal{Spec}\  l_0$ in that theory. Recall that $V$ is called semistable (with respect to $L$) if for every coherent subsheaf $W$ of $V$, we have $\mu(W)\leq \mu(V)$ and it is
called strongly semistable if $F^{n,*}_Y V $ is semistable for all $n\geq 0$.

In general, there exists a filtration 
\[
 0=V_0\subseteq V_1\subseteq \dots \subseteq V_{r-1}\subseteq V_r=V 
\]
of $V$ by subsheaves, such that the quotients $V_i/V_{i-1}$ are all semistable and such that the slopes $\mu(V_i/V_{i-1})$ are strictly decreasing for $i\geq 1$. 
This filtration is unique and is called the Harder-Narasimhan (HN) filtration of $V$. We will say that $V$ has a strongly semistable HN filtration if all the quotients
$V_i/V_{i-1}$ are strongly semistable. We shall write
\[
 \mu_{\textnormal{min}}(V):= \mu(V_r/V_{r-1})
\]
and
\[
 \mu_{\textnormal{max}}(V):=\mu(V_1).
\]
By the very definition of HN filtration, we have:
 \[
  V \text{ is semistable } \ \Leftrightarrow \ \mu_{\textnormal{min}}(V)= \mu_{\textnormal{max}}(V). 
 \]
An important consequence of the definitions is the following fact: if $V$ and $W$ are two torsion free sheaves on $Y$ and $\mu_{\textnormal{min}}(V) > 
\mu_{\textnormal{max}}(W)$,
then $\textnormal{Hom}_Y(V,W)=0$.


For more on the theory of semistable sheaves, see the monograph \cite{huybre}.

The following two theorems are key results from Langer 
\begin{theo}[{{\cite[Thm.\ 2.7]{langersem}}}] \label{lan}
 If $V$ is a torsion free coherent sheaf on $Y$, then there exists $n_0\geq 0$ such that $F^{n,*}_Y V$ has a strongly semistable HN filtration for all $n\geq n_0$. 
\end{theo}
If $V$ is a torsion free coherent sheaf on $Y$, we now define 
\[
 \bar{\mu}_{\textnormal{min}}(V):=\lim_{r\rightarrow\infty}\mu_{\textnormal{min}}(F^{r,*}_Y V) / \textnormal{char}(l_0)^r
\]
and
\[
  \bar{\mu}_{\textnormal{max}}(V):=\lim_{r\rightarrow\infty}\mu_{\textnormal{max}}(F^{r,*}_Y V) / \textnormal{char}(l_0)^r.
\]
Note that Theorem \ref{lan} implies that the two sequences $\mu_{\textnormal{min}}(F^{r,*}_Y V) / \textnormal{char}(l_0)^r$ and 
$\mu_{\textnormal{max}}(F^{r,*}_Y V) / \textnormal{char}(l_0)^r$ become constant when $r$ is sufficiently large, so the above definitions of $\bar{\mu}_{\textnormal{min}} $ 
and $\bar{\mu}_{\textnormal{max}} $ make sense.
Furthermore the sequences $\mu_{\textnormal{min}}(F^{r,*}_Y V) / \textnormal{char}(l_0)^r$ and $\mu_{\textnormal{max}}(F^{r,*}_Y V) / \textnormal{char}(l_0)^r$ are respectively
weakly 
decreasing and weakly increasing, therefore we have 
\[
  \mu_{\text{min}}(V)\geq \bar\mu_{\text{min}}(V)\  \text{ and } \ \bar \mu_{\text{max}}(V)\geq \mu_{\text{max}}(V). 
\]
Let us define \[\alpha(V):=\max\left\{\mu_{\text{min}}(V)- \bar\mu_{\text{min}}(V),\ \bar \mu_{\text{max}}(V)-\mu_{\text{max}}(V)\right\}.\]
\begin{theo}[{{\cite[Cor.\ 6.2]{langersem}}}]\label{langerstima}
If $V$ is of rank $r$, then
 \[
  \alpha(V)\leq\frac{r-1}{\textnormal{char}(l_0)} \max\left\{ \bar \mu_{\textnormal{max}}(\Omega_{Y/l_0}),0\right\}.
 \]
In particular, if $ \bar \mu_{\textnormal{max}}(\Omega_{Y/l_0})\geq 0$ and $\textnormal{char}(l_0)\geq d=\textnormal{dim} \ Y$,
\[
 \bar \mu_{\textnormal{max}}(\Omega_{Y/l_0})\leq \frac{\textnormal{char}(l_0)}{\textnormal{char}(l_0)+1-d}  \ \mu_{\textnormal{max}}(\Omega_{Y/l_0}).
\]
\end{theo}
We conclude this section with the following two lemmas from R\"ossler.
\begin{lemma}[{{\cite[Lemma 3.8]{rosstrongly}}}]\label{centr}
 Suppose that there is a closed $l_0$-immersion $i: Y \hookrightarrow B$, where $B$ is an abelian variety over $l_0$. Suppose that $\textnormal{Stab}_B(Y)=0$. Then
 $\Omega_Y^\vee$ is globally generated and for any dominant proper morphism $\phi:Y_0 \rightarrow Y$, where $Y_0$ is integral, we have $H^0(Y_0, \phi^*\Omega_Y^\vee)=0 $. 
 Furthermore,
 we have $\bar\mu_{\textnormal{min}}(\Omega_Y)>0$.
\end{lemma}
\begin{lemma}[{{\cite[Cor.\ 3.11]{rosstrongly}}}] \label{moret}
 Let $V$ be a finite rank, locally free sheaf over $Y$. Suppose that \\
 - for any surjective finite map $\phi: Y' \rightarrow Y$ with $Y' $ integral, we have $H^0(Y', \phi^*V)=0$, \\
 - $V^\vee$ is globally generated. \\
 Then  $H^0\left(Y, F^{n,*}_Y V\otimes \Omega_{Y/l_0}\right)=0$ for  $n$ sufficiently big. 
 
 Furthermore, 
 let
  $T\rightarrow Y$ be a torsor under 
  $V\left(F_{Y}^{n_0,*}V\right),$ where $n_0$ satisfies $H^0\left(Y, F^{n,*}_Y V\otimes \Omega_{Y/l_0}\right)=0$  for all $n>n_0$.
  Let $\phi:Y'  \rightarrow Y$ be a finite surjective morphism 
 and suppose that $Y'$ is integral. Then we have the implication: 
 
 $\phi^* T$ is a trivial $V\left(\phi^*(F^{n_0, *}_YV)\right)$-torsor $\Longrightarrow T$ is a trivial  
 $V\left(F_{Y}^{n_0,*}V\right)$-torsor.
\end{lemma}
The main ingredient of the proof of Lemma \ref{moret} is a result by Szpiro and Lewin-M\'en\'egaux which we will need later: 
\begin{prop}[{{\cite[exp. 2, Prop.\ 1]{pinceaux}}}]\label{lewin} If $V$ is a  vector bundle  over $Y$ such that
 $ H^0\left(Y, F^*_YV\otimes \Omega_{Y/l_0}\right)=0$, then 
 the map
\[
H^1(Y,V)\rightarrow H^1\left(Y,F^*_Y V\right)
\]
is injective.
\end{prop}
\section{Sparsity of $p$-divisible unramified liftings}
In this section we  prove our result on the sparsity of p-divisible unramified liftings (cf. Theorem \ref{torre} below).

Let $K$, $A$, $X$ and $U$ be as fixed in Notations and let $\text{Stab}_A(X)$ be trivial.
The construction of
the stabilizer commutes with the base change, so we have
\[
 \textnormal{Stab}_A(X)=\textnormal{Stab}_{\mathcal{A}}(\mathcal{X})\times_U \textnormal{Spec} K.
\]
Since $\textnormal{Stab}_A(X)$ is trivial, by generic flatness and finiteness, we can restrict the map 
$\pi: \textnormal{Stab}_{\mathcal{A}} (\mathcal{X})\rightarrow U$ to the inverse image of a non-empty open subscheme
$U'\subset U$ to obtain a finite flat commutative group scheme  of degree one \[\pi_{|\pi^{-1}(U')}: \pi^{-1}(U')\rightarrow U'.\] 
This implies that $\pi_{|\pi^{-1}(U')} $ is an isomorphism
and for any $\mathfrak{q}\in U'$ we have that 
$ \textnormal{Stab}_{A_{\mathfrak{q}^0}}(X_{\mathfrak{q}^0})$ 
is trivial. 
We will denote by $\tilde U\subseteq U$ the nonempty open subscheme 
\[
 \tilde U:=\left\{ \mathfrak{q}\in  U | \  \textnormal{Stab}_{A_{\mathfrak{q}^0}}(X_{\mathfrak{q}^0}) \ \textnormal{is trivial} \right\}.
\]
For any $\mathfrak{p}\in U$ we denote by $F_{\bar{k(\mathfrak{p})}}$ the Frobenius endomorphism on $\bar{k(\mathfrak{p})}$ and by $F_{R_1}$ the endomorphism 
of $R_1$ induced by $F_{\bar{k(\mathfrak{p})}}$ by functoriality.
We define 
\begin{align*}
 X_{\mathfrak{p}^0}'&:= X_{\mathfrak{p}^0}\times_{F_{\bar{k(\mathfrak{p})}}} \bar{k(\mathfrak{p})}  \\
 X_{\mathfrak{p}^1}'&:= X_{\mathfrak{p}^1}\times_{F_{R_1}} R_1 
\end{align*}
and we write
\[F_{X_{\mathfrak{p}^0}/\bar{k(\mathfrak{p})}}:X_{\mathfrak{p}^0}\rightarrow X_{\mathfrak{p}^0}' \]
for the relative Frobenius on $X_{\mathfrak{p}^0}$. For brevity's sake, from now on we will write \[\Omega_{X_{\mathfrak{p}^0}} \ \  \left(\text{resp. } \Omega_{X_{\mathfrak{p}^0}'},                                                                                                                                                                                        
\ \Omega_{X_{\mathfrak{p}^1}}, \
\Omega_{X_{\mathfrak{p}^1}'},\ \Omega_{X}\right)\]
instead of \[\Omega_{X_{\mathfrak{p}^0}/\bar{k(\mathfrak{p})}}\ \ \left(\text{resp. } \Omega_{X_{\mathfrak{p}^0}'/\bar{k(\mathfrak{p})}},\  \Omega_{X_{\mathfrak{p}^1}/R_1}, 
 \ \Omega_{X_{\mathfrak{p}^1}'/R_1}, \ 
\Omega_{X/K}\right).\]
Observe that since $ U$ is normal,  $\mathcal{A}$ is projective over $U$ (cf. Th.XI 1.4 in \cite{rayfai}).
Therefore there exists  a $U-$very ample line bundle $L$ on $\mathcal{X}$. 
For any $\mathfrak{p}\in U$ different from the generic point $\xi$, let us
denote by $L_{\mathfrak{p}}$ the inverse image of $L$ on $X_{\mathfrak{p}^0}$. Similarly we denote by $L_\xi$ the inverse image of $L$ on $X$.
From now on, for any vector bundle $G_{\mathfrak{p}}$ over $X_{\mathfrak{p}^0}$, we will  write 
$\text{deg}(G_\mathfrak{p})$ for the degree of $G_\mathfrak{p}$
with respect to $L_{\mathfrak{p}}$. Analogously, if $G_\xi$ is a vector bundle over $X$, we will write $\text{deg}(G_\xi)$ for the degree of $G_\xi$
with respect to $L_{\xi}$. Now consider the vector bundle $\Omega_{\mathcal{X}/U} $ over $\mathcal{X}$. For any natural number $m$, the map from $U$ to $\matZ$ defined by
\[
 \mathfrak{p}\mapsto  \chi\left((\Omega_{\mathcal{X}/U}\otimes L^m)_\mathfrak{p}\right)=\chi\left(\Omega_{X_{\mathfrak{p}^0}}\otimes L_{\mathfrak{p}}^m\right)
\]
and 
\[
 \xi\mapsto  \chi\left((\Omega_{\mathcal{X}/U}\otimes L^m)_\xi\right)=\chi\left(\Omega_{X}\otimes L_{\xi}^m\right)
\]
(here $\chi$ refers to the Euler characteristic)
is  constant on $U$ (cf. Ch.II, Sec.5 in \cite{mumab}). Therefore we have the equality
\[
 \chi\left(\Omega_{X_{\mathfrak{p}^0}}\otimes L_{\mathfrak{p}}^m\right)=\chi\left(\Omega_{X}\otimes L_{\xi}^m\right)
\]
for all $m\in \matN$ and for all $\mathfrak{p}\in U$. In other words, the Hilbert polynomial of $\Omega_{X_{\mathfrak{p}^0}} $ with respect to 
$L_{\mathfrak{p}} $ coincides with the Hilbert polynomial of $\Omega_{X} $ with respect to $L_{\xi} $. Since the degree of a vector bundle we defined at the 
beginning of this section
can be described in terms of its Hilbert polynomial (cf. Definition 1.2.11 in \cite{huybre}), we obtain that
 for every $\mathfrak{p}\in U$ we have $\text{deg}\left(\Omega_{X_{\mathfrak{p}^0}}\right)=
\text{deg}(\Omega_X)$.

 The following  lemma is a fundamental step to prove our sparsity Theorem \ref{torre}. 
\begin{lemma} \label{p}
 Let $K$, $A$, $X$ and $U$ be as fixed in Notations,  let $\textnormal{Stab}_A(X)$ be trivial and let $n$ be the dimension of $X$ over $K$. Then 
 \[
  \textnormal{Hom}_{X_{\mathfrak{p}^0}}\left(F_{X_{\mathfrak{p}^0}}^{k,*} \Omega_{X_{\mathfrak{p}^0}}, \Omega_{X_{\mathfrak{p}^0}} \right)=0
 \]
for any $k\geq 1$ and any $\mathfrak{p}\in \tilde U$ above a prime $p> n^2 \textnormal{deg}\left(\Omega_{X}\right) $.
\end{lemma}
\begin{proof}
 Let us notice first that, if $n=1$, then $X$ is a curve of genus $g$ at least $2$ and 
 \[
  \textnormal{Hom}_{X_{\mathfrak{p}^0}}\left(F_{X_{\mathfrak{p}^0}}^{k,*} \Omega_{X_{\mathfrak{p}^0}}, \Omega_{X_{\mathfrak{p}^0}} \right)=0
 \]
 is a simple consequence of the fact 
 \[
  \deg\left(F_{X_{\mathfrak{p}^0}}^{k,*} \Omega_{X_{\mathfrak{p}^0}}\right)=p^k(2g-2)>2g-2=\deg \Omega_{X_{\mathfrak{p}^0}}.
 \]
To treat the general case, 
let us fix  $\mathfrak{p}\in \tilde U$ above a prime $p> n^2 \textnormal{deg}\left(\Omega_{X}\right)$. We know that if 
 \[
  \mu_{\text{min}}\left(F_{X_{\mathfrak{p}^0}}^{k,*} \Omega_{X_{\mathfrak{p}^0}}\right)>\mu_{\text{max}}\left( \Omega_{X_{\mathfrak{p}^0}}\right)
 \]
then $\textnormal{Hom}_{X_{\mathfrak{p}^0}}\left(F_{X_{\mathfrak{p}^0}}^{k,*} \Omega_{X_{\mathfrak{p}^0}}, \Omega_{X_{\mathfrak{p}^0}} \right)=0$. Since
$ \mu_{\text{min}}\geq \bar\mu_{\text{min}}$ and $\bar \mu_{\text{max}}\geq \mu_{\text{max}} $, it is sufficient to show that, for every $k\geq 1$
\begin{equation}\label{bar}
  \bar\mu_{\text{min}}\left(F_{X_{\mathfrak{p}^0}}^{k,*} \Omega_{X_{\mathfrak{p}^0}}\right)>\bar \mu_{\text{max}}\left( \Omega_{X_{\mathfrak{p}^0}}\right).
 \end{equation}

%
Since $\text{Stab}_{A_{\mathfrak{p}^0}}(X_{\mathfrak{p}^0})$ is trivial, we can apply Lemma \ref{centr} to obtain $\bar\mu_{\text{min}}\left( \Omega_{X_{\mathfrak{p}^0}}\right)>0$.
In particular $\mu_{\text{min}}\left( \Omega_{X_{\mathfrak{p}^0}}\right)>0$ and $\text{deg}\left(\Omega_{X_{\mathfrak{p}^0}}\right)>0 $.
Using this and the equality $\bar\mu_{\text{min}}\left(F_{X_{\mathfrak{p}^0}}^{k,*} \Omega_{X_{\mathfrak{p}^0}}\right)=p^k
\bar\mu_{\text{min}}\left( \Omega_{X_{\mathfrak{p}^0}}\right)$, we see that (\ref{bar}) is implied by 
\begin{equation}\label{star}
 p\bar\mu_{\text{min}}\left( \Omega_{X_{\mathfrak{p}^0}}\right)>\bar\mu_{\text{max}}\left( \Omega_{X_{\mathfrak{p}^0}}\right).
\end{equation}
Theorem \ref{langerstima} gives us the following inequality

\[
 p\bar\mu_{\text{min}}\left( \Omega_{X_{\mathfrak{p}^0}}\right)\geq p\mu_{\text{min}}\left( \Omega_{X_{\mathfrak{p}^0}}\right)+(1-n) 
 \bar\mu_{\text{max}}\left( \Omega_{X_{\mathfrak{p}^0}}\right).
\]
so that  (\ref{star}) is  satisfied if
\begin{equation}\label{zt}
  p\mu_{\text{min}}\left( \Omega_{X_{\mathfrak{p}^0}}\right)>n  \bar\mu_{\text{max}}\left( \Omega_{X_{\mathfrak{p}^0}}\right).
\end{equation}
Since  $p> n^2 \textnormal{deg}\left(\Omega_{X}\right)\geq n, $ we can apply the second part of Theorem \ref{langerstima}   \[
                                       \bar\mu_{\text{max}}\left( \Omega_{X_{\mathfrak{p}^0}}\right)\leq \frac{p}{p+1-n}  
                                       \mu_{\text{max}}\left( \Omega_{X_{\mathfrak{p}^0}}\right),
                                      \]
 so that inequality (\ref{zt}) is implied by
\begin{equation}\label{john}
 (p+1-n)\mu_{\text{min}}\left( \Omega_{X_{\mathfrak{p}^0}}\right)>n \mu_{\text{max}}\left( \Omega_{X_{\mathfrak{p}^0}}\right).
\end{equation}

If $\Omega_{X_{\mathfrak{p}^0}}$ is semistable, (\ref{john}) gives $p>2n-1 $. Otherwise, we can  estimate  
$\mu_{\text{max}}\left( \Omega_{X_{\mathfrak{p}^0}}\right)$ and $\mu_{\text{min}}\left( \Omega_{X_{\mathfrak{p}^0}}\right) $ in the following way.
We know that 
\[ 
\mu_{\text{max}}\left( \Omega_{X_{\mathfrak{p}^0}}\right)=\frac{\text{deg}(M)}{\text{rk}(M)}
\]
for some subsheaf $0\neq M \subsetneq  \Omega_{X_{\mathfrak{p}^0}}$. 
Therefore we have 
 $\mu_{\text{max}}\left( \Omega_{X_{\mathfrak{p}^0}}\right)\leq \text{deg}(M)$. Furthermore, $\deg \left( \Omega_{X_{\mathfrak{p}^0}}/M\right)>0$, since 
 $ \bar\mu_{\text{min}}\left( \Omega_{X_{\mathfrak{p}^0}}\right)>0 $. This and the additivity of the degree on short exact sequences gives us 
 \[
  \mu_{\text{max}}\left( \Omega_{X_{\mathfrak{p}^0}}\right)\leq \text{deg}(M)\leq \text{deg}\left(\Omega_{X_{\mathfrak{p}^0}}\right)-1.
 \]
 Similarly,  
 \[ 
\mu_{\text{min}}\left( \Omega_{X_{\mathfrak{p}^0}}\right)=\frac{\text{deg}(Q)}{\text{rk}(Q)}
\]
for some $Q$ quotient of $\Omega_{X_{\mathfrak{p}^0}}$, so $\mu_{\text{min}}\left( \Omega_{X_{\mathfrak{p}^0}}\right)\geq 1/n$.
Inequality (\ref{john}) is then implied by
\[
 p>n^2\text{deg}\left(\Omega_{X_{\mathfrak{p}^0}}\right)+(n-1-n^2).
\]
Since $n-1-n^2$ is always negative, we are reduced to $ p>n^2\text{deg}\left(\Omega_{X_{\mathfrak{p}^0}}\right)$. Now 
$\text{deg}\left(\Omega_{X_{\mathfrak{p}^0}}\right)$ is greater or equal to one, so
$n^2\text{deg}\left(\Omega_{X_{\mathfrak{p}^0}}\right)\geq 2n-1$ for any $n$. This ensures us that the condition 
\[
  p>n^2\text{deg}\left(\Omega_{X_{\mathfrak{p}^0}}\right)
\]
is sufficient to have  $\mu_{\text{min}}\left(F_{X_{\mathfrak{p}^0}}^{k,*} \Omega_{X_{\mathfrak{p}^0}}\right)>\mu_{\text{max}}\left( \Omega_{X_{\mathfrak{p}^0}}\right)$
for every $k\geq1$ whether $\Omega_{X_{\mathfrak{p}^0}}$ is semistable or not. To conclude it is enough to remember that 
$\text{deg}\left(\Omega_{X_{\mathfrak{p}^0}}\right)$ coincides with $\text{deg}\left(\Omega_{X}\right)$.
\end{proof}
\begin{cor}
 The map \[
          H^1\left(X_{\mathfrak{p}^0},F_{X_{\mathfrak{p}^0}}^*\Omega_{X_{\mathfrak{p}^0}}^{\vee}\right)\rightarrow
          H^1\left(X_{\mathfrak{p}^0},F_{X_{\mathfrak{p}^0}}^{k,*}\Omega_{X_{\mathfrak{p}^0}}^{\vee}\right)
         \]
is injective for every $k\geq 1$ and every $\mathfrak{p}\in \tilde U$ above a prime $p> n^2\textnormal{deg}\left(\Omega_{X}\right)$.
\end{cor}
\begin{proof}
  Lemma \ref{p} and Proposition \ref{lewin} imply that
  \[
          H^1\left(X_{\mathfrak{p}^0},F_{X_{\mathfrak{p}^0}}^{h,*}\Omega_{X_{\mathfrak{p}^0}}^{\vee}\right)\rightarrow
          H^1\left(X_{\mathfrak{p}^0},F_{X_{\mathfrak{p}^0}}^{h+1,*}\Omega_{X_{\mathfrak{p}^0}}^{\vee}\right)
         \]
is injective for every $h\geq 0$. Therefore the composition
\begin{center}
\begin{tikzpicture}[description/.style={fill=white,inner sep=2pt}, ciao/.style={fill=red,inner sep=2pt}]
\matrix (m) [matrix of math nodes, row sep=3.5em,
column sep=1.5em, text height=1.5ex, text depth=0.25ex]
{    H^1\left(X_{\mathfrak{p}^0},F_{X_{\mathfrak{p}^0}}^*\Omega_{X_{\mathfrak{p}^0}}^{\vee}\right)  & 
 H^1\left(X_{\mathfrak{p}^0},F_{X_{\mathfrak{p}^0}}^{2,*}\Omega_{X_{\mathfrak{p}^0}}^{\vee}\right) & ... &
  H^1\left(X_{\mathfrak{p}^0},F_{X_{\mathfrak{p}^0}}^{k,*}\Omega_{X_{\mathfrak{q}^0}}^{\vee}\right) \\
  };
	\path[right hook->](m-1-1) edge (m-1-2);
	\path[right hook->](m-1-2) edge (m-1-3);
	\path[right hook->](m-1-3) edge (m-1-4);
\end{tikzpicture}
\end{center}
is an injective map.
\end{proof}
We are now ready to prove our sparsity result.
\begin{theo} \label{torre}
With the same hypotheses as in Lemma \ref{p},
 for any $\mathfrak{p}\in \tilde U$ above a prime $p> n^2\textnormal{deg}\left(\Omega_{X}\right)$,
the set
\[
 \left\{ P\in X_{\mathfrak{p}^0}(R_0) \ | \ P \textnormal{ lifts to an element of }  p A_{\mathfrak{p}^1}(R_1)\cap
 X_{\mathfrak{p}^1}(R_1)\right\}
\]
is not Zariski dense in $X_{\mathfrak{p}^0}$.
\end{theo}

\begin{proof}
 Let us fix $\mathfrak{p}$ as in the hypotheses. Since \[\textnormal{Crit}^1(\mathcal{X},\mathcal{A})(R_0)= p A_{\mathfrak{p}^1}(R_1)\cap
 X_{\mathfrak{p}^1}(R_1),\] we have that 
 \[
  \left\{ P\in X_{\mathfrak{p}^0}(R_0) \ | \ P \text{ lifts to an element of }  p A_{\mathfrak{p}^1}(R_1)\cap
 X_{\mathfrak{p}^1}(R_1)\right\}
 \]
 coincides with the image of $\textnormal{Crit}^1(\mathcal{X},\mathcal{A})(R_0)\rightarrow  X_{\mathfrak{p}^0}
 (R_0)$. 
 
Let us assume by contradiction that this image is dense in $ X_{\mathfrak{p}^0}
 (R_0)$. This implies that $\pi_1:\textnormal{Gr}_{1}(X_{\mathfrak{p}^{1}}) \rightarrow  X_{\mathfrak{p}^0}$ is a trivial torsor: the argument we use to show this is
 taken from R\"ossler (cf. beginning of the proof of Theorem 2.2 in \cite{rosstrongly}). First of all the closed map 
$
 \text{Crit}^1(\mathcal{X},\mathcal{A}) \rightarrow X_{\mathfrak{p}^0}
$
is surjective and so  we can choose an irreducible component
\[\text{Crit}^1(\mathcal{X},\mathcal{A})_0 \hookrightarrow
\text{Crit}^1(\mathcal{X},\mathcal{A})\]
which dominates $X_{\mathfrak{p}^0}$. 
Lemma \ref{centr} and Lemma \ref{p} allow us to apply the second part of Lemma \ref{moret} with 
$V=\Omega^\vee_{X_{\mathfrak{p}^0}}$, $Y=X_{\mathfrak{p}^0}$,  $n_0=1$, $T=\textnormal{Gr}_{1}(X_{\mathfrak{p}^{1}}) $  and $\phi$ equal to 
$ \text{Crit}^1(\mathcal{X},\mathcal{A})_0 \rightarrow X_{\mathfrak{p}^0}. $
We have that $\phi^*\textnormal{Gr}_{1}(X_{\mathfrak{p}^{1}})$ is trivial as $V\left(\phi^*F^{*}_{X_{\mathfrak{p}^0}}\Omega_{X_{\mathfrak{p}^0}}^\vee \right) $-torsor, since 
$\text{Crit}^1(\mathcal{X},\mathcal{A})_0$ is contained in $ \textnormal{Gr}_{1}(X_{\mathfrak{p}^{1}}).$
Hence  $\pi_1:\textnormal{Gr}_{1}(X_{\mathfrak{p}^{1}}) \rightarrow  X_{\mathfrak{p}^0}$ is trivial as 
$V\left(F^{*}_{X_{\mathfrak{p}^0}}\Omega_{X_{\mathfrak{p}^0}}^\vee\right)$-torsor. 
Let us take a section
$\sigma: X_{\mathfrak{p}^0} \rightarrow \textnormal{Gr}_{1}(X_{\mathfrak{p}^{1}})$.
By definition of Greenberg transform,
the map $\sigma$ over $R_0 $ corresponds to a map $\bar\sigma:\mathbb{W}_1(X_{\mathfrak{p}^0})\rightarrow X_{\mathfrak{p}^1} $ over $R_1$. We can precompose 
$\bar\sigma$ with the morphism $t:X_{\mathfrak{p}^1}\rightarrow \mathbb{W}_1(X_{\mathfrak{p}^0}) $ corresponding to
\begin{align*}
 W_1(\mathcal{O}_{X_{\mathfrak{p}^0}})&\rightarrow \mathcal{O}_{X_{\mathfrak{p}^1}} \\
 (a_0,a_{1}) &\mapsto \tilde a_0^p+\tilde a_1 p 
\end{align*}
where $\tilde a_i$ lifts $a_i$. 
Consider now the following diagram 
\begin{center}
\begin{tikzpicture}[description/.style={fill=white,inner sep=2pt}, ciao/.style={fill=red,inner sep=2pt}]
\matrix (m) [matrix of math nodes, row sep=3.5em,
column sep=3.5em, text height=1.5ex, text depth=0.25ex]
{   X_{\mathfrak{p}^1}  & \mathbb{W}_1( X_{\mathfrak{p}^0}) & X_{\mathfrak{p}^1} \\
    X_{\mathfrak{p}^0} &  X_{\mathfrak{p}^0} & X_{\mathfrak{p}^0} \\
  };
	\path[->,font=\scriptsize]
		(m-2-1) edge node[auto] {$  $} (m-1-1)
		(m-2-1) edge node[auto] {$F_{X_{\mathfrak{p}^0}}  $} (m-2-2)
		(m-2-2) edge node[auto] {$  $} (m-1-2)
		(m-1-2) edge node[auto] {$\bar \sigma  $} (m-1-3)
		(m-2-3) edge node[auto] {$  $} (m-1-3)
		(m-2-2) edge node[auto] {$ \text{Id} $} (m-2-3)
		;
	\path[->,font=\scriptsize] 
		(m-1-1) edge node[above] {$ t $} (m-1-2)
                ;
\end{tikzpicture}
\end{center}
Its left square  is commutative, since the composition 
\begin{center}
\begin{tikzpicture}[description/.style={fill=white,inner sep=2pt}, ciao/.style={fill=red,inner sep=2pt}]
\matrix (m) [matrix of math nodes, row sep=3.5em,
column sep=3.5em, text height=1.5ex, text depth=0.25ex]
{    X_{\mathfrak{p}^0} & X_{\mathfrak{p}^1}  &  \mathbb{W}_1(X_{\mathfrak{p}^0}) \\
  };
	\path[->,font=\scriptsize] 
		(m-1-1) edge node[above] {$ $} (m-1-2)
		(m-1-2) edge node[above] {$ t$} (m-1-3)
                ;
\end{tikzpicture}
\end{center}
simply corresponds to the map
\begin{align*}
 W_1(\mathcal{O}_{X_{\mathfrak{p}^0}})&\rightarrow \mathcal{O}_{X_{\mathfrak{p}^0}} \\
 (a_0,a_{1}) &\mapsto  a_0^p.
\end{align*}
For the commutativity of the right square, notice that by the very definition of the 
transition morphism $\pi_1:\textnormal{Gr}_{1}(X_{\mathfrak{p}^{1}}) \rightarrow  X_{\mathfrak{p}^0}$ we have a commutative diagram
\begin{center}
\begin{tikzpicture}[description/.style={fill=white,inner sep=2pt}, ciao/.style={fill=red,inner sep=2pt}]
\matrix (m) [matrix of math nodes, row sep=3.5em,
column sep=3.5em, text height=1.5ex, text depth=0.25ex]
{  \textnormal{Hom}_{R_1}(\mathbb{W}_1( X_{\mathfrak{p}^0}), X_{\mathfrak{p}^1}) & 
\textnormal{Hom}_{R_0}(X_{\mathfrak{p}^0},\textnormal{Gr}_{1}( X_{\mathfrak{p}^1})) \\
   \textnormal{Hom}_{R_0}(X_{\mathfrak{p}^0},X_{\mathfrak{p}^0}) &  \\
  };
	\path[->,font=\scriptsize]
		(m-1-1) edge node[left] {$ \textnormal{reduction mod. }p $} (m-2-1)
		(m-1-2) edge node[auto] {$ (\pi_1\circ -) $} (m-2-1)
		;
	\path[->,font=\scriptsize] 
		(m-1-1) edge node[above] {$  $} (m-1-2)
                ;
\end{tikzpicture}
\end{center}
In particular, $\text{Id}_{X_{\mathfrak{p}^0}}=\pi_1\circ \sigma= (\textnormal{reduction mod. }p)(\bar \sigma),$
which is exactly what we wanted to verify.

We obtain therefore that
$\bar\sigma \circ t:X_{\mathfrak{p}^1}\rightarrow  X_{\mathfrak{p}^1}$ is a lift of the Frobenius $ F_{X_{\mathfrak{p}^0}} $.

The diagram below is also commutative
\begin{center}
\begin{tikzpicture}[description/.style={fill=white,inner sep=2pt}, ciao/.style={fill=red,inner sep=2pt}]
\matrix (m) [matrix of math nodes, row sep=3.5em,
column sep=3.5em, text height=1.5ex, text depth=0.25ex]
{   X_{\mathfrak{p}^1}  & \mathbb{W}_1( X_{\mathfrak{p}^0}) & X_{\mathfrak{p}^1} \\
   \text{Spec}(R_1)  &  \text{Spec}(R_1) & \text{Spec}(R_1) \\
  };
	\path[->,font=\scriptsize]
		(m-1-1) edge node[auto] {$  $} (m-2-1)
		(m-2-1) edge node[auto] {$F_{R_1}  $} (m-2-2)
		(m-1-2) edge node[auto] {$  $} (m-2-2)
		(m-1-2) edge node[auto] {$\bar \sigma  $} (m-1-3)
		(m-1-3) edge node[auto] {$  $} (m-2-3)
		(m-2-2) edge node[auto] {$ \text{Id} $} (m-2-3)
		;
	\path[->,font=\scriptsize] 
		(m-1-1) edge node[above] {$ t $} (m-1-2)
                ;
\end{tikzpicture}
\end{center}
In fact, by definition, $\bar \sigma$ is a morphism over 
$R_1  $, so the right square is commutative. 
The commutativity of the left square is easy to check, since we know explicitely $t$ and $F_{R_1}$.
Therefore $\bar\sigma \circ t $ is a lift of the Frobenius $ F_{X_{\mathfrak{p}^0}} $ compatible with $F_{R_1} $: this implies the existence of a morphism of $R_1$-schemes
\[
 \tilde F: X_{\mathfrak{p}^1}\rightarrow  X_{\mathfrak{p}^1}' 
\]
lifting the relative Frobenius $F_{X_{\mathfrak{p}^0}/R_0}$. 

As shown in part (b) of the proof of Th\'eor\`eme 2.1 in \cite{illusie},
since 
the image of 
$\tilde F^*:\Omega_{X_{\mathfrak{p}^1}'}\rightarrow \tilde F_*
\Omega_{X_{\mathfrak{p}^1}}$ is contained in $p \tilde F_*
\Omega_{X_{\mathfrak{p}^1}}$ and the multiplication by $p$ induces an isomorphism $p:F_{X_{\mathfrak{p}^0}/R_0,*}
\Omega_{X_{\mathfrak{p}^0}}\xrightarrow{\sim} p \tilde F_*
\Omega_{X_{\mathfrak{p}^1}},$ there exists a unique map
\[
 f:=p^{-1}\tilde F^*:\Omega_{X_{\mathfrak{p}^0}'}\rightarrow F_{X_{\mathfrak{p}^0}/R_0,*}
\Omega_{X_{\mathfrak{p}^0}}
\]
making the diagram below commutative
\begin{center}
\begin{tikzpicture}[description/.style={fill=white,inner sep=2pt}, ciao/.style={fill=red,inner sep=2pt}]
\matrix (m) [matrix of math nodes, row sep=3.5em,
column sep=3.5em, text height=1.5ex, text depth=0.25ex]
{  \Omega_{X_{\mathfrak{p}^1}'}  & p \tilde F_*
\Omega_{X_{\mathfrak{p}^1}} \\
   \Omega_{X_{\mathfrak{p}^0}'}  &  F_{X_{\mathfrak{p}^0}/R_0,*}
\Omega_{X_{\mathfrak{p}^0}}  \\
  };
	\path[->,font=\scriptsize]
		(m-2-1) edge node[auto] {$f  $} (m-2-2)
		(m-2-2) edge node[auto] {$ p $} (m-1-2)
		;
	\path[->>,font=\scriptsize]
	(m-1-1) edge node[auto] {$  $} (m-2-1);
	\path[->,font=\scriptsize] 
		(m-1-1) edge node[above] {$ \tilde F^* $} (m-1-2)
                ;
\end{tikzpicture}
\end{center}
Proposition 3 in \cite{xin} states that  the adjoint of $f$
\[
 \bar f:  F_{X_{\mathfrak{p}^0}}^*\Omega_{X_{\mathfrak{p}^0}}=F_{X_{\mathfrak{p}^0}/R_0}^*\Omega_{X_{\mathfrak{p}^0}'}\rightarrow
\Omega_{X_{\mathfrak{p}^0}}
\]
is generically bijective.
This clearly contradicts Lemma \ref{p}. 
\end{proof}

\section{The number of irreducible components of the critical scheme of complete intersections}
In this last section we provide an upper bound for the number of irreducible components of the critical scheme 
$\text{Crit}^1(\mathcal{X},\mathcal{A}) $  in the case in which $X$ is a smooth complete intersection.

Let $A/K$ be an abelian variety of dimension $n$ and let $L$ be a very ample line bundle on 
$A$. Let $c\in \matN$ be positive and  let $H_1$, $H_2, ...,H_c \in 
|L|$ be general. We define $X:=H_1 \cap H_2 \cap ... \cap H_c$.
Suppose that $X$ is smooth. 

Let us take a sufficiently small open $V\subseteq \text{Spec}(\mathcal{O}_K) $ 
such that $A$ extends over $V$ to an abelian scheme $\mathcal{A}$, $L$ extends
to a $V$-very ample line bundle $\mathcal{L}$, $H_i$ 
extends to $\mathcal{H}_i$ for every $i$ and $\mathcal{X}:=\mathcal{H}_1\cap \mathcal{H}_2\cap ...\cap \mathcal{H}_c$ is smooth.  
We can restrict $V$ if necessary and suppose 
$K/\matQ$ is unramified at $\mathfrak{p}$.
%
\begin{theo} \label{eff}
 Let $K$ be a number field, $A/K$ be an abelian variety of dimension $n$ and let $L$ be a very ample line bundle on 
$A$. Let $c\in \matN$ be positive and  let $H_1$, $H_2, ...,H_c \in 
|L|$ be general. Suppose that $X:=H_1 \cap H_2 \cap ... \cap H_c$ is smooth.
If $\mathfrak{p}$ is in the open subscheme 
$V$ defined above,
then the number of irreducible components of $\textnormal{Crit}^1(\mathcal{X},\mathcal{A}) $ is bounded by 
\[
 p^{2n}\left(\sum_{h=0}^{n-c}{2n-2c\choose h}{c \choose n-c-h}p^{n-c-h} \right) (L^n)^2.
\]
\end{theo}
(Here $ (L^n)$ denotes the intersection number of $L$)
\begin{proof}
To obtain Theorem \ref{eff}, we follow step by step Buium's approach to prove the Manin-Mumford conjecture for curves (cf. \cite{buiumvol}, Theorem 1.11): we first show that
$\textnormal{Crit}^1(\mathcal{X},\mathcal{A}) $
can be realized as the intersection of two projective varieties (see $\mathbb{P}(E_X) $ and $ \textnormal{[p]}_*\text{Gr}_1(A_{\mathfrak{p}^1})$ below) 
and then use the product of their degrees to bound the number of its irreducible components.
Since $X$ is not  necessarily of dimension one, 
the computation of the degree of $\mathbb{P}(E_X) $ is slightly more demanding here
than the correspondent one in Buium's work.

Let us fix $\mathfrak{p}\in V$. 
The torsor
    $\text{Gr}_1(X_{\mathfrak{p}^1})\rightarrow X_{\mathfrak{p}^0}$, (resp. $\text{Gr}_1(A_{\mathfrak{p}^1})\rightarrow A_{\mathfrak{p}^0})$)
    corresponds
to an element $\eta_X\in H^1\left(X_{\mathfrak{p}^0},F^{*}_{X_{\mathfrak{p}^0}}\Omega_{X_{\mathfrak{p}^0}}^\vee\right)$,   (resp. 
$\eta_A\in H^1\left(A_{\mathfrak{p}^0},F^{*}_{A_{\mathfrak{p}^0}}\Omega_{A_{\mathfrak{p}^0}}^\vee\right)).$
Under the natural isomorphisms  \\

{\centering
  $ \displaystyle
    \begin{aligned} 
     H^1\left(X_{\mathfrak{p}^0},F^{*}_{X_{\mathfrak{p}^0}}\Omega_{X_{\mathfrak{p}^0}}^\vee\right)&\simeq \textnormal{Ext}^1
\left(F^{*}_{X_{\mathfrak{p}^0}}\Omega_{X_{\mathfrak{p}^0}},
 \mathcal{O}_{X_{\mathfrak{p}^0}}\right) \\ 
H^1\left(A_{\mathfrak{p}^0},F^{*}_{A_{\mathfrak{p}^0}}\Omega_{A_{\mathfrak{p}^0}}^\vee\right)&\simeq \textnormal{Ext}^1
\left(F^{*}_{A_{\mathfrak{p}^0}}\Omega_{A_{\mathfrak{p}^0}},
 \mathcal{O}_{A_{\mathfrak{p}^0}}\right)
    \end{aligned}
  $
\par} $ $ \\
\normalsize{$\eta_X$ and} $\eta_A$  correspond to extensions
of vector bundles \\

{\centering
  $ \displaystyle
    \begin{aligned} 
     0\rightarrow \mathcal{O}_{X_{\mathfrak{p}^0}}&\rightarrow E_X \rightarrow F^{*}_{X_{\mathfrak{p}^0}}\Omega_{X_{\mathfrak{p}^0}} \rightarrow 0 \\
  0\rightarrow \mathcal{O}_{A_{\mathfrak{p}^0}}&\rightarrow E_A \rightarrow F^{*}_{A_{\mathfrak{p}^0}}\Omega_{A_{\mathfrak{p}^0}} \rightarrow 0.
    \end{aligned}
  $
\par} $ $ \\ 
\normalsize{For} any locally free sheaf  $W$ over a base $S$ of finite type over a field, we  shall write  $\mathbb{P}(W)$ for the projective bundle associated to  $W$, i.e 
the $S$-scheme representing the functor on $S$-schemes
\begin{align*}
 T\longmapsto \{ \text{iso. classes of surjective morphisms of }&\mathcal{O}_T\textnormal{-modules }W_T\rightarrow Q, \ \\
 &\textnormal{where } Q \ \textnormal{ is locally free of rank 1}\}.
\end{align*}

As showed in paragraph 1 of \cite{md}, the two extensions above give us two divisors \\

{\centering
  $ \displaystyle
    \begin{aligned} 
    D_X&:=\mathbb{P}\left(F^{*}_{X_{\mathfrak{p}^0}}\Omega_{X_{\mathfrak{p}^0}}\right)\subseteq \mathbb{P}(E_X) \\
    D_A&:=\mathbb{P}\left( F^{*}_{A_{\mathfrak{p}^0}}\Omega_{A_{\mathfrak{p}^0}}\right)\subseteq \mathbb{P}(E_A)
    \end{aligned}
  $
\par} $ $ \\ 
belonging respectively to the linear systems $|\mathcal{O}_{\mathbb{P}(E_X)}(1)|$ and $|\mathcal{O}_{\mathbb{P}(E_A)}(1)|$, and  \\

{\centering
  $ \displaystyle
    \begin{aligned} 
    \text{Gr}_1(X_{\mathfrak{p}^1})&\simeq \mathbb{P}(E_X)\setminus D_X \\
 \text{Gr}_1(A_{\mathfrak{p}^1})&\simeq \mathbb{P}(E_A)\setminus D_A.
    \end{aligned}
  $
\par} $ $ \\
If $i$ denotes the closed immersion $i:X_{\mathfrak{p}^0} \rightarrow A_{\mathfrak{p}^0} $,
then it is not difficult to show that there is a natural restriction homomorphism 
$i^* E_A \rightarrow E_X $ prolonging the homomorphism $i^*\Omega_{A_{\mathfrak{p}^0}} \rightarrow  \Omega_{X_{\mathfrak{p}^0}}$. The homomorphism 
$i^* E_A \rightarrow E_X $ is clearly surjective, so it induces a closed immersion $j:\mathbb{P}(E_X)\hookrightarrow \mathbb{P}(E_A)$ prolonging 
$ \text{Gr}_1(X_{\mathfrak{p}^1})\hookrightarrow  \text{Gr}_1(A_{\mathfrak{p}^1})$.
Therefore we have a commutative diagram
\footnotesize{
\begin{center}
\begin{tikzpicture}[description/.style={fill=white,inner sep=2pt}, ciao/.style={fill=red,inner sep=2pt}]
\matrix (m) [matrix of math nodes, row sep=3.5em,
column sep=3.5em, text height=1.5ex, text depth=0.25ex]
{ & \textnormal{[p]}_*\text{Gr}_1(A_{\mathfrak{p}^1}) \\
\text{Gr}_1(X_{\mathfrak{p}^1}) & \text{Gr}_1(A_{\mathfrak{p}^1}) \\  
   \mathbb{P}(E_X) & \mathbb{P}(E_A) \\
   X_{\mathfrak{p}^0}  &  A_{\mathfrak{p}^0} \\
  };
	\path[->,font=\scriptsize]		
		(m-3-1) edge node[auto] {$ \pi_X $} (m-4-1)
		(m-3-2) edge node[auto] {$ \pi_A $} (m-4-2)
		;
		\path [>=latex,->,font=\scriptsize] (m-1-2) edge [bend left=60] node[auto] {$T$} (m-4-2);
	\path[right hook->,font=\scriptsize]
	        (m-2-2) edge node[auto] {$  $} (m-3-2)
	        (m-2-1) edge node[auto] {$  $} (m-3-1)
	        (m-1-2) edge node[auto] {$  $} (m-2-2)
		(m-2-1) edge node[above] {$  $} (m-2-2)
		(m-3-1) edge node[above] {$ j $} (m-3-2)
		(m-4-1) edge node[above] {$ i $} (m-4-2)
                ;
\end{tikzpicture}
\end{center}
}
\normalsize{
Let} us denote by $\mathcal{L}_{\mathfrak{p}}$  the base change of $\mathcal{L}$ to $A_{\mathfrak{p}^0} $. It is standard to prove that  
\[
 \mathcal{H}:=\pi_A^*\mathcal{L}_{\mathfrak{p}}\otimes\mathcal{O}_{\mathbb{P}(E_A)}(1)
\]
is very ample on $\mathbb{P}(E_A)$ (cf. pag 4 in \cite{buiumvol}). 
We have
\[
 \mathcal{H}|_{\mathbb{P}(E_X)}=\pi_X^* i^*\mathcal{L}_{\mathfrak{p}}\otimes \mathcal{O}_{\mathbb{P}(E_X)}(1)
\]
\[
 \mathcal{H}|_{[p]_*\text{Gr}_1(A_{\mathfrak{p}^1})}=T^*\mathcal{L}_{\mathfrak{p}}
\]
since $D_A\in|\mathcal{O}_{\mathbb{P}(E_A)}(1)| $ and $[p]_*\text{Gr}_1(A_{\mathfrak{p}^1})\subseteq \text{Gr}_1(A_{\mathfrak{p}^1})\simeq \mathbb{P}(E_A)\setminus D_A$.
We know that $[p]_*\text{Gr}_1(A_{\mathfrak{p}^1})$ is the maximal abelian subvariety of  $\text{Gr}_1(A_{\mathfrak{p}^1}) $ and we also know that
the multiplication by $p$ map on $\text{Gr}_1(A_{\mathfrak{p}^1}) $ factors through the isogeny $T$. This implies that 
$T$ has degree at most $p^{2n}$, so we have the following estimate
\[
 \deg_{\mathcal{H}}([p]_*\text{Gr}_1(A_{\mathfrak{p}^1}))\leq p^{2n} (\mathcal{L}_{\mathfrak{p}}^n).
\]
Let us now consider $\deg_{\mathcal{H}}(\mathbb{P}(E_X))$. It coincides with
\begin{equation}\label{boh}
 \int_{\mathbb{P}(E_X)}c_1( \mathcal{H}|_{\mathbb{P}(E_X)})^{2n-2c}
\end{equation}
where $c_1$ stands for the first Chern class in the Chow ring and $ \int_{\mathbb{P}(E_X)}$ stands for the push-forward morphism to $ \text{Spec}(R_0) $ in the Chow theory.
Since \[c_1( \mathcal{H}|_{\mathbb{P}(E_X)})=c_1(\pi_X^* i^*\mathcal{L}_{\mathfrak{p}})+ c_1( \mathcal{O}_{\mathbb{P}(E_X)}(1))\] we can re-write (\ref{boh}) as
\[
 \int_{\mathbb{P}(E_X)}\sum_{h=0}^{2n-2c}{2n-2c\choose h} c_1(\pi_X^* i^*\mathcal{L}_{\mathfrak{p}})^h \cdot c_1( \mathcal{O}_{\mathbb{P}(E_X)}(1))^{2n-2c-h}.
\]
Equivalently
\[
\int_{X_{\mathfrak{p}^0}}\sum_{h=0}^{2n-2c}{2n-2c\choose h} c_1( i^*\mathcal{L}_{\mathfrak{p}})^h \cdot \pi_{X,*}\left(c_1( \mathcal{O}_{\mathbb{P}(E_X)}(1))^{2n-2c-h}\right)
\]
and by definition of Segre class this is
\[
 \int_{X_{\mathfrak{p}^0}}\sum_{h=0}^{2n-2c}{2n-2c\choose h} c_1( i^*\mathcal{L}_{\mathfrak{p}})^h \cdot s_{n-c-h}(E^\vee_X). 
\]
Notice that the Segre classes of the  dual of $E_X$ appear in our formula: this is due to the fact that we are not using Fulton's geometric notation 
for the projective bundle associated to a vector bundle (cf. Note at the end of B.5.5 in \cite{ful}).
Since $s_k=0$ if $k<0$,  we end up with
\[
\int_{X_{\mathfrak{p}^0}}\sum_{h=0}^{n-c}{2n-2c\choose h} c_1( i^*\mathcal{L}_{\mathfrak{p}})^h \cdot s_{n-c-h}(E^\vee_X).
\]
Now the exact sequence
\[
  0\rightarrow \mathcal{O}_{X_{\mathfrak{p}^0}}\rightarrow E_X \rightarrow F^{*}_{X_{\mathfrak{p}^0}}\Omega_{X_{\mathfrak{p}^0}} \rightarrow 0
\]
implies
\[
 s_{n-c-h}(E^\vee_X)=
 s_{n-c-h}\left(F^{*}_{X_{\mathfrak{p}^0}}\Omega^\vee_{X_{\mathfrak{p}^0}}\right)
\]
and so
\[
 s_{n-c-h}(E^\vee_X)=
 p^{n-c-h}s_{n-c-h}(\Omega^\vee_{X_{\mathfrak{p}^0}})
\]
(here we have used the following fact: the pullback of a cycle $\eta$ of codimension $j$ through the Frobenius map coincides with $p^j \eta$). 
Therefore we have to study the following sum
\begin{equation} \label{er}
 \sum_{h=0}^{n-c}{2n-2c\choose h}p^{n-c-h} c_1( i^*\mathcal{L}_{\mathfrak{p}})^h \cdot s_{n-c-h}(\Omega^\vee_{X_{\mathfrak{p}^0}}).
\end{equation}
The short exact sequence
\[
 0\rightarrow \Omega_{X_{\mathfrak{p}^0}}^{\vee}\rightarrow  i^*\Omega_{A_{\mathfrak{p}^0}}^{\vee}
 \rightarrow N\rightarrow 0
\]
(where $N$ is the normal bundle for $i$)
gives
\[
 c_t(\Omega^\vee_{X_{\mathfrak{p}^0}})c_t(N)=c_t(i^*\Omega^\vee_{A_{\mathfrak{p}^0}})=1,
\]
so that 
$ c_t(N)=s_t(\Omega^\vee_{X_{\mathfrak{p}^0}}).$
Recalling that 
\[
 c_t(N)=\left(1+c_1(i^*\mathcal{L}_{\mathfrak{p}})t\right)^c
\]
we obtain
\[
 s_{n-c-h}(\Omega^\vee_{X_{\mathfrak{p}^0}})=c_{n-c-h}(N) 
 ={c \choose n-c-h} c_1(i^*\mathcal{L}_{\mathfrak{p}})^{n-c-h}.
\]
Substituting in (\ref{er}), we obtain
\begin{align*}
  \left(\sum_{h=0}^{n-c}{2n-2c\choose h}  {c \choose n-c-h}p^{n-c-h}\right)c_1(i^*\mathcal{L}_{\mathfrak{p}})^{n-c}.
\end{align*}
Therefore $\deg_{\mathcal{H}}(\mathbb{P}(E_X))$ is
\[
   \left(\sum_{h=0}^{n-c}{2n-2c\choose h}  {c \choose n-c-h}p^{n-c-h}\right)
 \int_{X_{\mathfrak{p}^0}} c_1(i^*\mathcal{L}_{\mathfrak{p}})^{n-c}.
\]
Since  $X_{\mathfrak{p}^0}=H_{1,\mathfrak{p}}\cap ...\cap
H_{c,\mathfrak{p}}$ where 
 $H_{1,\mathfrak{p}},...,H_{c,\mathfrak{p}}$ belong to $|\mathcal{L}_{\mathfrak{p}}|$, we have
\[
\int_{X_{\mathfrak{p}^0}} c_1(i^*\mathcal{L}_{\mathfrak{p}})^{n-c}= \int_{A_{\mathfrak{p}^0}}c_1(\mathcal{L}_{\mathfrak{p}})^n=(\mathcal{L}_{\mathfrak{p}}^n)
\]
and
\[
 \deg_{\mathcal{H}}(\mathbb{P}(E_X))=\left(\sum_{h=0}^{n-c}{2n-2c\choose h}{c \choose n-c-h}p^{n-c-h}\right) (\mathcal{L}_{\mathfrak{p}}^n).
\]
Now B\'ezout's theorem in Fulton's form (cf. \cite{ful}, p. 148) says that the number of irreducible components in the intersection of two projective varieties of degrees
$d_1,d_2$ cannot exceed $d_1d_2$. In particular, the number of irreducible components of $\textnormal{Crit}^1(\mathcal{X},\mathcal{A}) $ is less or equal than
\[
  p^{2n}\left(\sum_{h=0}^{n-c}{2n-2c\choose h}{c \choose n-c-h}p^{n-c-h}\cdot \right) (\mathcal{L}_{\mathfrak{p}}^n)^2.
\]

\normalsize{Notice that $ (\mathcal{L}_{\mathfrak{p}}^n)=(L^n)$, by the same reasoning done before  Lemma \ref{p}.

}
\end{proof}
\begin{rema}
 One can consider any intersection
 $X:=H_1 \cap H_2 \cap ... \cap H_c $ where $H_i\in |L_i|$ for some very ample line bundles $L_i$.
 In this more general case, the computations in our proof become a bit more complex, but it is still  possible to give an explicit bound for
 the number of irreducible components of $\textnormal{Crit}^1(\mathcal{X},\mathcal{A}) $. We have
 \[
  c_j(N)=\sum_{1\leq i_1<...<i_j\leq c} \ \prod_{k=i_1}^{i_j}  c_1(i^*\mathcal{L}_{k,\mathfrak{p}})
 \]
which implies
\begin{equation*} \label{citan}
 s_{n-c-h}(\Omega^\vee_{X_{\mathfrak{p}^0}})=\sum_{1\leq i_1<...<i_{n-c-h}\leq c} \ \prod_{k=i_1}^{i_{n-c-h}}  c_1(i^*\mathcal{L}_{k,\mathfrak{p}}).
\end{equation*}
Therefore, defining $\mathcal{H}:=\pi_A^*\mathcal{L}_{1,\mathfrak{p}}\otimes\mathcal{O}_{\mathbb{P}(E_A)}(1) $, then $\deg_{\mathcal{H}}(\mathbb{P}(E_X))$ is
\[
  \sum_{h=0}^{n-c}{2n-2c\choose h}p^{n-c-h}\sum_{1\leq i_1<...<i_{n-c-h}\leq c}  \left( \int_{X_{\mathfrak{p}^0}} c_1(i^*\mathcal{L}_{1,\mathfrak{p}})^{h}
  \prod_{k=i_1}^{i_{n-c-h}}c_1(i^*\mathcal{L}_{k,\mathfrak{p}})\right).
\]
We have
\[
 \int_{X_{\mathfrak{p}^0}} c_1(i^*\mathcal{L}_{1,\mathfrak{p}})^{h}
  \prod_{k=i_1}^{i_{n-c-h}}c_1(i^*\mathcal{L}_{k,\mathfrak{p}}) 
 =I_{i_1,...,i_{n-c-h}}
\]
where $I_{i_1,...,i_{n-c-h}}$ is the following intersection number
\[
  I_{i_1,...,i_{n-c-h}}:=(\overbrace{L_1\cdots L_1}^\text{h+1 times} \cdot L_2 \cdot L_3 \cdots L_c
  \cdot L_{i_1}\cdot L_{i_2}\cdots L_{i_{n-c-h}}).
\]
We obtain that $\deg_{\mathcal{H}}(\mathbb{P}(E_X))$ is 
\[
 \sum_{h=0}^{n-c}{2n-2c\choose h}p^{n-c-h}\sum_{1\leq i_1<...<i_{n-c-h}\leq c}  
 I_{i_1,...,i_{n-c-h}}.
\]
and therefore the number of irreducible components of $\textnormal{Crit}^1(\mathcal{X},\mathcal{A}) $ is bounded by:
\[
 p^{2n} (L_1^n)\sum_{h=0}^{n-c}{2n-2c\choose h}p^{n-c-h}\sum_{1\leq i_1<...<i_{n-c-h}\leq c}  
 I_{i_1,...,i_{n-c-h}}.
\]

\end{rema}
\section*{Acknowledgements}
This paper constitutes part of my PhD thesis. I would like to thank Prof. D. R\"ossler
for having introduced me to this topic and for his invaluable mentoring advice.

I am also indebted to the anonymous referee for his/her useful remarks.

This work was partly supported by the SFB 1085 “Higher Invariants” at the
University of Regensburg.
\bibliography{bibliotatevol}{}
\bibliographystyle{alpha}

\end{document}